\documentclass[reqno,T1,11pt]{amsproc}
\usepackage{amssymb}
\usepackage{amsmath}
\usepackage{amsfonts}
\usepackage{multicol}
\usepackage[canadian]{babel}
\usepackage{xcolor}
\usepackage{geometry}
\usepackage{xifthen} % provides \isempty
\usepackage{xparse} % new commands with more than one optional argument
\usepackage{fixltx2e} % for \textsubscript
\usepackage{comment}

\newcommand{\uwlim}{\qopname\relax m{uwlim}}

\usepackage{tikz}
\usetikzlibrary{arrows,cd,decorations.markings,shapes.geometric,shapes}

\definecolor{myurlcolor}{rgb}{0,0,0.4}
\definecolor{mycitecolor}{rgb}{0,0.5,0}
\definecolor{myrefcolor}{rgb}{0.5,0,0}
\usepackage[pagebackref,draft=false]{hyperref}
\hypersetup{colorlinks,
linkcolor=myrefcolor,
citecolor=mycitecolor,
urlcolor=myurlcolor}

\renewcommand{\H}{\mathcal{H}}	% Hilbert space
\newcommand{\B}[1]{\mathcal{B}(#1)}	% bounded operators
\newcommand{\beq}{\begin{equation}}
\newcommand{\eeq}{\end{equation}}
\newcommand{\N}{\mathbb{N}}

\newcommand{\C}{\mathbb{C}}
\newcommand{\op}{\mathrm{op}}
\newcommand{\eps}{\varepsilon}

\newcommand{\cat}[1]{\mathsf{#1}}		% generic category

\newcommand{\Ban}{\mathsf{Ban}}			% (complex) Banach spaces
\newcommand{\Hilb}{\mathsf{Hilb}} 		% Hilb, cat of Hilbert spaces
\newcommand{\Rep}[1]{\mathsf{Rep}(#1)}		% cat normal Hilbert space reps of a C$^*$-algebra
\newcommand{\NRep}[1]{\mathsf{nRep}(#1)}	% cat of normal Hilbert space reps of a W$^*$-algebra
\newcommand{\id}[1]{1_{#1}}			% identity morphism
\swapnumbers
\theoremstyle{plain}
\newtheorem{thm}{Theorem}[section]
\newtheorem{lem}[thm]{Lemma}
\newtheorem{prop}[thm]{Proposition}
\newtheorem{cor}[thm]{Corollary}

\newtheorem{defn}[thm]{Definition}

\theoremstyle{remark}
\newtheorem{ex}[thm]{Example}
\newtheorem{rem}[thm]{Remark}

\numberwithin{equation}{section}

\allowdisplaybreaks

\pgfdeclarelayer{edgelayer}
\pgfdeclarelayer{nodelayer}
\pgfsetlayers{edgelayer,nodelayer,main}
\tikzstyle{none}=[inner sep=0pt]
\tikzstyle{simple}=[-,draw=black,line width=1.000]
\tikzstyle{bn}=[circle,inner sep=2pt,fill=black,draw=black,line width=0.8 pt]

\usepackage{todonotes}

\begin{document}
\sloppy

% vertical spacing in multiline equations
\setlength{\jot}{6pt}

%-------------------------------------------------------------------

%%%%%%%%%%%% title page stuff %%%%%%%%%%%%%%%%%%%%%%%%%%

\title[Infinite direct sums in C$^*$-categories and W$^*$-categories]{The universal property of infinite direct sums in C$^*$-categories and W$^*$-categories}

\author{Tobias Fritz}
\author{Bas Westerbaan}

\address{Perimeter Institute for Theoeretical Physics, Waterloo, Ontario, Canada}
\email{tfritz@pitp.ca}

\address{Institute for Computing and Information Sciences, Radboud Universiteit Nijmegen, The Netherlands}
\email{bas@westerbaan.name}



\thanks{\textit{Acknowledgements.} We thank Robert Furber for copious
help and discussion, as well as Chris Heunen and Martti Karvonen
for discussion. Part of this work was conducted while the first
author was at the Max Planck Institute for Mathematics in the
Sciences in Leipzig, Germany.}

\begin{abstract}
When formulating universal properties for objects in a dagger category, one usually expects a universal property to characterize the universal object up to unique \emph{unitary} isomorphism. We observe that this is automatically the case in the important special case of C$^*$-categories, provided that one uses enrichment in Banach spaces. We then formulate such a universal property for infinite direct sums in C$^*$-categories, and prove the equivalence with the existing definition due to Ghez, Lima and Roberts in the case of W$^*$-categories. These infinite direct sums specialize to the usual ones in the category of Hilbert spaces, and more generally in any W$^*$-category of normal representations of a W$^*$-algebra.

Finding a universal property for the more general case of direct integrals remains an open problem.
\end{abstract}

\maketitle

\section{Introduction}

Among the most central and useful concepts of category theory is the notion of universal property. Its utility relates to the fact that a universal property characterizes an object up to unique isomorphism, and to the vast number of mathematical construction which can be defined or characterized in terms of universal properties. Among these is the concept of \emph{direct sum} of two objects $A_1$ and $A_2$ in an additive category $\cat{C}$. It can be defined in several equivalent ways via the following standard result, where $\cat{Ab}$ is the category of abelian groups.

\begin{thm}[{e.g.~\cite[Proposition~1.2.4]{borceux}}]
\label{directsum_equiv_trad}
Let $\cat{C}$ be an additive category and $A_1,A_2\in\cat{C}$ two objects. Then the following pieces of structure are equivalent for an object $A\in\cat{C}$:
\begin{enumerate}
	\item\label{coprod} $A$ represents the functor
		\[
			\cat{C} \longrightarrow \cat{Ab}, \qquad B \longmapsto \cat{C}(A_1, B) \times \cat{C}(A_2, B)
		\]
    \item\label{prod} $A$ represents the functor
		\[
			\cat{C}^\op \longrightarrow \cat{Ab}, \qquad B \longmapsto \cat{C}(B, A_1) \times \cat{C}(B, A_2)
		\]
	\item\label{biprod} There are morphisms $p_i : A \to A_i$ and $\kappa_i : A_i \to A$ satisfying the \emph{biproduct equations}
	\[
		p_i \kappa_j = \delta_{ij}, \qquad \kappa_1 p_1 + \kappa_2 p_2 = 1.
	\]
\end{enumerate}
\end{thm}

Here, \ref{coprod} and~\ref{prod} state that $A$ is a coproduct and product of $A_1$ and $A_2$, respectively. \ref{biprod} is a purely equational characterization which also proves that direct sums are preserved by every additive functor.

There appear many useful universal properties for purely
algebraic structures.  Treating analytical structures
categorically is just as straightforward in some cases, but very challenging
in many other cases. It usually involves equipping the categories
under consideration themselves with certain analytical structures.
An interesting case of where this happens is the theory of
\emph{W$^*$-categories}, which model categories of Hilbert space
representations of C$^*$-algebras or W$^*$-algebras. The most
paradigmatic example of a W$^*$-category being $\Hilb$, the category
of Hilbert spaces with bounded linear maps.

In this paper, we study \emph{infinite} direct sums of objects in W$^*$-categories and in the more general case of C$^*$-categories. While our definitions and results also apply to the case of finite direct sums, they reduce to the standard Theorem~\ref{directsum_equiv_trad} in this case. As the most basic example, the Hilbert space $\ell^2(I)$ for a (finite or) infinite set $I$ is the direct sum of $I$ many copies of $\C$. Part of the \emph{raison d'\^{e}tre} for analytic structure on an algebraic object is that it often makes the objects under consideration more well-behaved. This is indeed what happens in our case as well: while Theorem~\ref{directsum_equiv_trad} does not generalize to infinite direct sums, since infinite coproducts and infinite products do no longer coincide in a general additive category (even when they exist), our main result (Theorem~\ref{directsum_equiv}) provides such a generalization for infinite direct sums in W$^*$-categories.

\subsection*{Summary} After giving some background on C$^*$-categories and W$^*$-categories in Section~\ref{background}, we start by investigating universal properties for objects in C$^*$-categories in Section~\ref{univ_props}. We find that defining universal properties in terms of representability of a functor enriched in Banach spaces provides a good notion of universal property in a C$^*$-category, since it guarantees uniqueness up to unique \emph{unitary} isomorphism of the representing object. We then introduce our definition of the particular universal property of infinite direct sums in C$^*$-categories in Section~\ref{sec_directsum_def}. In Section~\ref{wstarcatcase}, we specialize to the case of W$^*$-categories, and prove an analog of Theorem~\ref{directsum_equiv_trad} in this setting, showing that our infinite direct sums are equivalent to the direct sums in W$^*$-categories as introduced originally by Ghez, Lima and Roberts~\cite{wstarcat}. 

\subsection*{Direct integrals} We have also attempted to formulate a similar universal property for \emph{direct integrals} in W$^*$-categories, but we have not been able to make this idea work. As has been argued by Simon Henry\footnote{See \href{https://mathoverflow.net/questions/313158/what-does-it-mean-for-a-category-to-admit-direct-integrals}{mathoverflow.net/questions/313158/what-does-it-mean-for-a-category-to-admit-direct-integrals}.}, finding such a universal property requires the consideration of additional structure on the W$^*$-category. Our current best bet is that doing so will involve a notion of \emph{W$^*$-category with rigged objects}, generalizing rigged Hilbert spaces~\cite{vol4}.

\subsection*{C*-algebras and W*-algebras}  We give a brief recap of C*-algebras and W*-algebras for the benefit of readers who have a background in category theory but not in operator algebras, while referring to the literature for the details~\cite{sakai,kr,bram}.
 
\emph{C$^*$-algebras} are abstract versions of algebras of operators
on a Hilbert space which are closed in the operator norm and closed
under taking adjoint operators, where taking adjoints is encoded
in the ``star'' operation $*$. Similarly, \emph{W$^*$-algebras} are
abstract versions of algebras of operators on a Hilbert space which
are in addition closed under limits with respect to the strong
operator topology, a sort of pointwise converge.\footnote{A W$^*$-algebra
    of operators is usually called a von Neumann algebra.
    However, our references \cite{bas,bram} use the terms W$^*$-algebra and
        von Neumann algebra interchangeably.}
In fact, every abstract C$^*$-algebra or W$^*$-algebra is isomorphic
to one having such a concrete representation~\cite[\S30\textsubscript{XIV}, \S48\textsubscript{VII}]{bram}.
For example, every $*$-closed subalgebra
of the algebra of complex~$n\times n$-matrices is a C$^*$-algebra
and a~W$^*$-algebra. Moreover, every commutative unital C$^*$-algebra
is isomorphic to~$C(X)$ for some compact Hausdorff space~$X$~\cite[\S27]{bram}. And
every commutative W$^*$-algebra is isomorphic to~$L^\infty(X,\mu)$
for some complete finite measure space~$(X,\mu)$~\cite[\S53\textsubscript{XI}]{bram}.

\section{C$^*$-categories and W$^*$-categories}
\label{background}

A category can be thought of as a many-object version of a monoid.
This manifests itself in the fact that for any category $\cat{C}$,
every endomorphism homset $\cat{C}(A,A)$ is a monoid. In the same
way, a C$^*$-category is a many-object version of a unital
C$^*$-algebra. Indeed in a~C$^*$-category~$\cat{C}$, every
homset~$\cat{C}(A,A)$ will be a unital~C$^*$-algebra. Similarly, a
W$^*$-category $\cat{C}$ is a many-object version of a W$^*$-algebra,
and in a W$^*$-category every homset~$\cat{C}(A,A)$ carries the structure of
a~W$^*$-algebra.

    The following precise definitions are due to Ghez, Lima and Roberts \cite{wstarcat}.

\begin{defn}
    Let~$\cat{C}$ be a category together with an endofunctor~$*\,\colon \cat{C} \to \cat{C}^\op$,
    whose action on arrows we denote by~$a\mapsto a^*$. The pair $(\cat{C},*)$ is called a \emph{complex~$*$-category} if the following conditions hold.
\begin{enumerate}
    \item $\cat{C}$ is enriched over complex vector spaces, meaning that:
	    \begin{enumerate}
		    \item every homset~$\cat{C}(A,B)$ is a complex vector space,
		    \item composition of morphisms is bilinear.
	   \end{enumerate}
    \item The contravariant endofunctor~$*$ is
	    \begin{enumerate}
		    \item identity-on-objects: $a^* : B \to A$ for $a : A \to B$,
		    \item involutive: $a^{**} = a$,
		    \item conjugate linear:
			    \[
				    (\lambda a)^* \ =\  \overline{\lambda} a^*,\qquad (a+b)^* \ =\  a^* + b^*
			    \]
			    for morphisms $a, b : A \to B$ and $\lambda\in\C$.
	    \end{enumerate}
    \item For every morphism~$a : A \to B$,
	    \begin{enumerate}
		    \item there is a morphism $b :  A \to A$ with~$a^*a=b^*b$,
		    \item $a^*a=0$ implies $a = 0$.
	    \end{enumerate}
\end{enumerate}
\end{defn}

It is easy to see that in a complex $*$-category $\cat{C}$, every endomorphism homset $\cat{C}(A,A)$ carries the structure of a complex $*$-algebra. 

\begin{defn}
A complex $*$-category $(\cat{C},*)$ is called a
\begin{enumerate}
	\item \emph{normed $*$-category} if each homset~$\cat{C}(A,B)$ is a normed space such that
	    \[
		    \| ab \| \ \leq\  \|a\| \, \|b\|
	    \]
                for any~$a\colon B \to C$ and~$b\colon A \to B$.
	\item\label{Cstarcat}
    \emph{C$^*$-category} if moreover every homset is norm-complete, and the C$^*$-identity
    \[
	\|a \|^2\  =\  \|a^*a\|
     \]
	holds for every morphism $a$.
        (This turns $\cat{C}(A,A)$ into a~C$^*$-algebra.)
\item\label{Wstarcat}
    \emph{W$^*$-category} if moreover
        where every homset has a Banach space predual.
        
	Equivalently~\cite[Prop.~2.15]{wstarcat}, a C$^*$-category~$\cat{C}$
            is a~W$^*$-category
            if every~$\cat{C}(A,A)$ is a W$^*$-algebra
            and~$\cat{C}(A,B)$ is self dual
            as a Hilbert~$\cat{C}(A,A)$-module
            with inner product~$\langle a,b\rangle \equiv a^*b$.
\end{enumerate}
As the homset~$\cat{C}(A,A)$
    in a C$^*$-category
    is a C$^*$-algebra,
    we will freely use definitions from~C$^*$-algebras on endomorphisms,
    so we say e.g.~that~$a\colon A \to A$ is a \emph{projection} if it is a projection as an element of the~C$^*$-algebra $\cat{C}(A,A)$.

Several notions extend naturally to all morphisms:
    a morphism~$a \colon A \to B$ is called an \emph{isometry} if~$a^*a=1$
        (and in that case~$aa^*$ is a projection)
        and it is a \emph{unitary} if additionally~$aa^*=1$.
\end{defn}

We again refer to~\cite{wstarcat} for more details on the following examples.

\begin{ex}
The category of Hilbert spaces with bounded linear maps between them
    is a W$^*$-category in the obvious way, meaning that $*$ takes each bounded linear map to its adjoint and the norm is the operator norm.
\end{ex}

\begin{ex}
\label{rep_ex}
If $A$ is a C$^*$-algebra, then the category of non-degenerate
representations $\pi : A \to \B{\H}$ on Hilbert spaces $\H$ is a
W$^*$-category. The morphisms~$(\H_1,\pi_1) \to (\H_2,\pi_2)$ are
exactly the \emph{intertwiners}, i.e.~the bounded linear maps~$f :
\H_1 \to \H_2$ such that the diagram
\[\begin{tikzcd}
	\H_1 \ar{r}{f} \ar[swap]{d}{\pi_1(a)} & \H_2 \ar{d}{\pi_2(a)} \\
	\H_1 \ar{r}{f} & \H_2
\end{tikzcd}\]
commutes for all $a\in A$. It is well-known that this is a W$^*$-category~\cite[Example~2.4]{wstarcat}, where the Banach space structure on each hom-set is the usual one on a space of operators using the operator norm, and the involution is given by the usual Hermitian adjoint. We denote this W$^*$-category by $\Rep{A}$. In the most basic case $A = \C$, we get $\Rep{\C} \cong \Hilb$, the category of Hilbert spaces.
\end{ex}

\begin{ex}
\label{nrep_ex}
If $N$ is a W$^*$-algebra, then the category of normal unital representations on Hilbert spaces is a W$^*$-category, with intertwiners as morphisms as above. We denote this W$^*$-category by $\NRep{N}$. For a C$^*$-algebra $A$ and $A^{**}$ its universal enveloping W$^*$-algebra, we have $\NRep{A^{**}} \cong \Rep{A}$, since non-degenerate representations of $A$ extend uniquely to normal unital representations of $A^{**}$.
\end{ex}

\section{Universal properties in C$^*$-categories}
\label{univ_props}

In this section, we write $\cat{C}$ for a C$^*$-category.
We start with a generalization of~\cite[Lemma 7]{westerbaan2016universal}
    to C$^*$-categories before discussing universal properties.

\begin{lem}
	Let $A,B\in\cat{C}$. Then for any~$a\colon A \to B$ with~$\|a\| \leq 1$
    and projections~$s \colon A \to A$,
        $t\colon B\to B$,
	the following are equivalent:
\begin{enumerate}
\item $a^* t a \leq 1-s$,
\item $a s a^* \leq 1 - t$,
\item $tas = 0$,
\item $sa^* t = 0$.
\end{enumerate}
\label{contrapositionlemma}
\end{lem}
\begin{proof}
We will first prove~$a^*t a \leq 1-s \:\Leftrightarrow\: tas=0$.
    Assume~$a^* t a \leq 1-s$.
    Then~$0 \leq s a^* t a s \leq s (1-s) s = 0$.
    Thus~$s a^* t a s = 0$, which is to say~$tas = 0$.
    Now, assume~$tas = 0$.
    Then~$ta = ta(1-s)$ and~$a^*t = (1-s)a^* t$.
    Hence~$a^* ta = (1-s)a^*ta(1-s) \leq 1-s$, as desired.
A similar argument proves~$as a^* \leq 1-t \Leftrightarrow sa^*t=0$.
To finish the proof, it is sufficient to
        show~$tas=0 \Leftrightarrow sa^*t=0$, which follows directly
        by applying~$(\ )^*$.
\end{proof}

\begin{lem}
Let~$a \colon A \leftrightarrows B \colon b$ be any two morphisms
    such that~$\|a \| \leq 1$, $\| b\| \leq 1$ and~$ab=1$.
    Then~$b$ is an isometry with~$a=b^*$.
\label{isometrylemma}
\end{lem}
\begin{proof}
    To start, note~$\|a ^*a\| = \|a\|^2 \leq 1$, so~$a^*a \leq 1$.
    Similarly~$b^*b \leq 1$.
    Combined~$1 = b^*a^*ab \leq b^*b \leq 1$, hence~$b^*b=1$
        and so~$b$ is an isometry.
    Similarly~$aa^*=1$.
    Now~$a (1-bb^*) a^* = aa^* - abb^*a^* = 1-1=0 \leq 0$.
    So~$a^* (1-0) a \leq bb^*$ by Lemma~\ref{contrapositionlemma}.
    Similarly~$bb^* \leq a^*a$. Hence~$bb^*=a^*a$.
    We are done: $a = aa^*a=abb^*=b^*$.
\end{proof}

We will need the full power of this observation only in the next section. However, the following special case is generally useful when working with universal properties in C$^*$-categories:

\begin{cor}
Let $u : A \to B$ be an invertible morphism with $\| u \| \leq 1$ and $\| u^{-1} \| \leq 1$. Then $u$ is unitary.
\label{unitaries}
\end{cor}

Now let $\Ban$ be the category of Banach spaces with bounded linear maps of norm $\leq 1$. Every C$^*$-category is $\Ban$-enriched. Our notion of universal property in a C$^*$-category is the standard one for enriched categories~\cite[Section~1.10]{kelly}, namely with enrichment in $\Ban$:

\begin{defn}
\label{def_univ}
A functor $F : \cat{C} \to \Ban$ is \emph{representable} if there is an object $A \in \cat{C}$ together with a natural isomorphism $\cat{C}(A,-)\cong F$ that is a componentwise isometric isomorphism of Banach spaces.
\end{defn}

Corollary~\ref{unitaries} and the (weak) enriched Yoneda lemma~\cite[Section~1.9]{kelly} imply:

\begin{cor}
\label{unique_up_to_unitary}
For any representable functor $\cat{C} \to \Ban$, the representing object is unique up to unique \emph{unitary} isomorphism
\end{cor}

\begin{proof}
	If both $A$ and $B$ represent $F$ via $\cat{C}(A,-)\cong F\cong \cat{C}(B,-)$, then the enriched Yoneda lemma gives us an isomorphism $u : A \to B$. Since the natural isomorphism between the hom-functors is required to be a componentwise isometry by $\Ban$-enrichment, we necessarily have $\| u \| = \| u^{-1} \| = 1$, making $u$ into a unitary by Corollary~\ref{unitaries}.
\end{proof}

This uniqueness up to unique unitary is a desideratum of a good notion of universal property for C$^*$-categories, and more generally for all dagger categories~\cite{daglims}. Corollary~\ref{unitaries} tells us that the compatibility with the dagger structure is automatic thanks to the enrichment in $\Ban$, so that we do not need to formulate our universal properties in the form of dagger limits as in~\cite{daglims}.

\section{The universal property of infinite direct sums}
\label{sec_directsum_def}

We now turn to the study of infinite direct sums by defining them in terms of a universal property as in Definition~\ref{def_univ}, and then proving that the existing notions of direct sum in the category of Hilbert spaces and more generally all W$^*$-categories indeed satisfy this universal property. While our considerations also apply to finite direct sums as a special case, in this case our discussion provides nothing new, since then our universal property reduces to the standard one of biproducts in an additive category.

We again let $\cat{C}$ be a C$^*$-category. Let $I$ be a finite or infinite index set. For a family $(A_i)_{i\in I}$ of objects in $\cat{C}$, we consider the functor
\[
	\bigoplus_i \cat{C}(A_i,-) \colon \  \cat{C} \to \Ban
\]
given by associating to every object $B$ the set
\beq
\label{AiB}
	\bigoplus_i \cat{C}(A_i , B) \ := \ \Bigl\{ (f_i : A_i \to B)_{i\in I} \Bigm\vert\textstyle  \sum_i f_i f_i^* < \infty \Bigr\}.
\eeq
Here, $\sum_i f_i f_i^*$ can be interpreted as converging in the ultraweak topology\footnote{For positive elements~$(x_i)_{i \in I}$ of a W$^*$-algebra, convergence and value of the sum~$\sum_{i \in I} x_i$ coincides in the ultraweak and ultrastrong topology and, in fact, equals the supremum of the partial sums.} in the case where $\cat{C}(B,B)$ is a W$^*$-algebra. If this is not the case, what we mean by the condition $\sum_i f_i f_i^* < \infty$ is that there is a fixed element which upper bounds all partial sums $\sum_{i\in S} f_i f_i^*$ for finite $S\subseteq I$, or equivalently that the sum converges ultraweakly in the double dual W$^*$-algebra $\cat{C}(B,B)^{**}$.

The set~\eqref{AiB} is a vector space under componentwise addition and scalar multiplication, where closure under addition follows from $f_i g^*_i + g_i f^*_i \leq f_i f^*_i + g_i g^*_i$, which holds because of $(f_i - g_i) (f^*_i - g^*_i) \geq 0$. We furthermore consider this vector space as a normed space under the norm
\[
	\| (f_i)_{i \in I} \| \ :=\  \Bigl\| \sum_i f_i f_i^* \Bigr\|^{1/2},
\]
where we interpret the sum as before either in terms of a supremum over the norm of finite partial sums, or as converging in the double dual $\cat{C}(B,B)^{**}$.
(This norm is in fact the one associated to the~$B$-valued inner
        product~$\langle (f_i)_i, (g_i)_i \rangle  :=
                \sum_i f_ig_i^*$,
                i.e.~$\| (f_i)_i \|^2 = \|\langle (f_i)_i, (f_i)_i \rangle\|$.
        The proof of the triangle inequality
                is similar to that for a
            norm associated to a regular inner product,
            using a generalized Cauchy--Schwarz inequality.
            See for instance~\cite[\S142\textsubscript{III--V}]{bas}
                or~\cite[Prop.~2.3]{paschke}.)

\begin{lem}
	Let $\cat{C}$ be a C$^*$-category, $(A_i)_{i\in I}\in\cat{C}$ a family of objects and $B\in\cat{C}$. Then $\bigoplus_i \cat{C}(A_i,B)$ is complete, and therefore a Banach space.
\label{lem_complete}
\end{lem}

%The proof generalizes the standard proof showing that $\ell^2(I)$ is a Hilbert space.

\begin{proof}
Suppose that we are given a Cauchy sequence of families $(f^{(n)})_{n\in\N} = ( ( f_i^{(n)} )_{i\in I} )_{n\in\N}$. For given $j$, the projection map
\[
	\bigoplus_i \cat{C}(A_i,B) \longrightarrow \cat{C}(A_j,B)	
\]
is bounded, and therefore $( f_j^{(n)} )_{n\in\N}$ is Cauchy as well; let us denote the limit by $f_j : A_j \to B$. We will now show that $f \in \bigoplus_i \cat{C}(A_i,B)$ for $f := (f_j)_{j\in I}$, which amounts to verifying the relevant square summability. The Cauchy assumption implies that the sequence $(f^{(n)})_{n\in\N}$ is uniformly bounded by some constant $C$, i.e.
\[
    \| f^{(n)} \| \ = \ 
	\Bigl\| \textstyle \sum_i f_i^{(n)} (f_i^{(n)})^* \Bigr\|^{1/2} \ \leq\  C.
\]
Restricting to a finite partial sum over $S\subseteq I$ and taking $n\to\infty$ shows that
\[
	\Bigl\| \textstyle \sum_{i\in S} f_i f_i^* \Bigr\|^{1/2} \ \leq\  C
\]
as well. Since $S$ was arbitrary, we conclude that $\| f \| \leq C$.

    The convergence $f^{(n)} \to f$ works similarly. By assumption, for every~$\eps > 0$ we have~$\left\| f^{(n)} - f^{(m)} \right\| \leq \eps$ for sufficiently large $n,m\in\N$. In particular, we know that for any finite $S \subseteq I$,
\[
	\Bigl\| \sum_{i\in S} \bigl(f_i^{(n)} - f_i^{(m)}\bigr) \bigl(f_i^{(n)} - f_i^{(m)}\bigr)^* \Bigr\|^{1/2} \ \leq\  \eps
\]
for sufficiently large $n,m\in\N$. Taking $m\to\infty$ gives
\[
	\Bigl\| \sum_{i\in S} \bigl(f_i^{(n)} - f_i\bigr) \bigl(f_i^{(n)} - f_i\bigr)^* \Bigr\|^{1/2} \ \leq \ \eps.
\]
Since $S$ was arbitrary, we conclude $\| f^{(n)} - f \|  \leq \eps$ for sufficiently large $n$.
\end{proof}

The map $B \mapsto \bigoplus_i \cat{C}(A_i,B)$ is functorial in $B$ in the obvious way: for $g : B \to C$, a family $(f_i : A_i \to B)_{i\in I}$ gets mapped to $(gf_i : A_i \to C)_{i\in I}$, which satisfies the relevant square integrability condition because of
\[
	\sum_i (g f_i)(g f_i)^* = g \left( \sum_i f_i f_i^* \right) g^* < \infty.
\]
Thus we get a functor
\[
	\bigoplus_i \cat{C}(A_i,-)  \colon \  \cat{C} \to \Ban.
\]
We now get to our main definition:

\begin{defn}
\label{directsumdef}
Let $\cat{C}$ be a C$^*$-category. An \emph{$I$-indexed direct sum} of the family of objects $(A_i)_{i\in I}$ is an object $\bigoplus_i A_i$ which represents the $\Ban$-enriched functor $\bigoplus_i \cat{C}(A_i,-)$.
\end{defn}

By Corollary~\ref{unique_up_to_unitary}, this characterizes infinite direct sums up to unique unitaries. Since $\cat{C}$ is a dagger category, the same holds dually, meaning that $\bigoplus_i A_i$ also represents the functor
\[
	\bigoplus_i \cat{C}(-,A_i) \: : \: \cat{C}^\op \to \Ban
\]
given by associating to every object $B$ the Banach space
\[
	\bigoplus_i \cat{C}(B, A_i) \ := \ 
    \Bigl\{\, \textstyle (f_i : B \to A_i)_{i\in I} \Bigm\vert \sum_i f_i^* f_i < \infty \,\Bigr\},
\]
where the sum is, as before, taken in the ultraweak topology of $\cat{C}(B,B)^{**}$ or to be understood in terms of uniform bounds on finite partial sums, and again considered a Banach space with respect to the norm
\[
	\| (f_i)_{i \in I} \|\  := \ \Bigl\| \,\sum_i f_i^* f_i \,\Bigr\|^{1/2}.
\]
This Banach space is contravariantly functorial in $B$ in the obvious way.

\begin{rem}
For finite $I$, our direct sums are also dagger limits in the sense of~\cite{daglims}. For infinite $I$, this is no longer the case, since infinite direct sums are not even limits; compare the degeneracy result~\cite[Theorem~5.2]{daglims}, which indicates that infinite dagger products cannot be expected to exist. In this sense, our notion of universal property via Definition~\ref{directsumdef} seems better adapted to the category of Hilbert spaces or other W$^*$-categories than that of~\cite{daglims}.
\end{rem}

\begin{rem}
	One may wonder whether our direct sums can be obtained as \emph{weighted limits} or \emph{weighted colimits}, which are generally regarded as the adequate notion of limits and colimits in an enriched context like ours. We believe this not to be the case: using enrichment in $\Ban$, there does not seem to be any weight which would give rise to the functor $\bigoplus_i \cat{C}(A_i,-)$ from~\eqref{AiB}. The main reason is that mere $\Ban$-enrichment does not have access to the involution, which however is a crucial piece of structure in~\eqref{AiB}.
\end{rem}

\begin{rem}
	We could also consider \emph{direct sums with kernel} in the following sense. Let $K = (K_{ij}) \in \C^{I\times I}$ be a positive semidefinite matrix. Then given a family of objects $(A_i)_{i\in I}$ as before, we can consider the functor
	\[
		\cat{C}\longrightarrow \Ban, \qquad B\ \longmapsto\ {\bigoplus_i}^K \cat{C}(A_i, B)
        \ := \ \Bigl\{\,  \textstyle (f_i : A_i \to B)_{i\in I} \Bigm\vert \sum_{i,j} K_{ij} f_i f_j^* < \infty \,\Bigr\},
	\]
	where the completeness proof with respect to the norm $f\mapsto \| \sum K_{ij} f_i f_j^* \|^{1/2}$ works in the same way as the proof of Lemma~\ref{lem_complete}.

	We can thus define the \emph{direct sum with kernel $K$} to be an object representing this functor in the same way as above. For example in $\Hilb$ with $A_i = \C$ for all $i\in I$, the resulting direct sum would be the space of families $(\xi_i)_{i\in I}$ with $\xi_i\in\C$ that are $K$-square summable in that $\sum_{i,j} K_{ij} \bar{\xi}_i \xi_j < \infty$, turned into a Hilbert space with respect to the inner product with kernel $K$. We do not yet know what these direct sums with kernel will amount to in other cases.
\end{rem}

\section{Direct sums in W$^*$-categories}
\label{wstarcatcase}

In a W$^*$-category, we characterize the existence of direct sums as follows:

\begin{thm}
\label{directsum_equiv}
Let $\cat{C}$ be a W$^*$-category and $(A_i)_{i\in I}$ a family of objects in $\cat{C}$. Then the following pieces of structure are equivalent for an object $A$:
\begin{enumerate}
\item\label{universalprop} $A$ is an $I$-indexed direct sum $\bigoplus_i A_i$.
\item\label{universalprop2} 
There is a family of
        morphisms~$\left(\kappa_j : A_j \to A\right)_{j\in I}$
        with~$\sum_{j \in I} \kappa_j \kappa_j^* < \infty$
        such that for any object~$B$
        and family of
        morphisms~$\left(f_j: A_j \to B\right)_{j\in I}$
        with~$\sum_{j \in I}f_j f_j^* < \infty$,
        there is a unique~$f\colon A \to B$
        such that $f \kappa_j = f_j$ for all~$j \in I$.
Moreover, such an~$f$ satisfies~$\| f \|^2 = \| \sum_{j\in I} f_j f_j^* \|$.
\item\label{biprodsum} There is a family of morphisms $\left(\kappa_j : A_j \to A\right)_{j\in I}$ such that $\kappa_i^*\kappa_j = \delta_{ij}$, and
\beq
\label{complete}
        \sum_{j\in I} \kappa_j \kappa_j^* \ =\  1_A.
\eeq
\end{enumerate}
\end{thm}

As far as we know, the first definition of infinite direct sums in W$^*$-categories has been given in~\cite[p.~100]{wstarcat} in the form~\ref{biprodsum}. From our perspective, it is preferable to replace this definition by a universal property as in Definition~\ref{directsumdef}, and to regard the original definition of Ghez, Lima and Roberts as an alternative characterization instead.

\begin{proof}
\ref{universalprop2} is a simple reformulation of the universal property \ref{universalprop} via the Yoneda lemma, as follows. Given a natural isometric isomorphism $\cat{C}(A,B) \cong \bigoplus_i \cat{C}(A_i,B)$, we define the family $(\kappa_i)$ to be the image of $\id{A}$. Then for any $f : A \to B$, the naturality diagram
\[\begin{tikzcd}
	\cat{C}(A,A) \ar{r}{\cong} \ar{d}{f\circ -} & \bigoplus_i \cat{C}(A_i,A) \ar{d}{(f\circ -)} \\
	\cat{C}(A,B) \ar{r}{\cong} & \bigoplus_i \cat{C}(A_i,B)
\end{tikzcd}\]
shows that the universal property takes $f$ to the family $(f_j)
:= (f\kappa_j)$. By the definition of the norm on $\bigoplus_i
\cat{C}(A_i,B)$ and the assumption that the universal property is
an isometric isomorphism, we indeed have $\| f \|^2 = \| \sum_j
f_j f_j^* \|$, which also gives $\sum_j \kappa_j \kappa_j^*
< \infty$ in the case $f = \id{A}$. Conversely, it is clear that
any family~$(\kappa_j)$ as in~\ref{universalprop2} gives an isometric
isomorphism~$\cat{C}(A,B) \cong \bigoplus_i \cat{C}(A_i,B)$ that
is natural in $B$.

We now derive \ref{biprodsum} from \ref{universalprop2}. Applying the assumption on norms with $(f_j) := (\kappa_j)$ gives $\| \sum_i \kappa_i \kappa_i^* \| = \| \id{A} \|$, and therefore $\| \kappa_j \| \leq 1$ for all $j\in I$. For fixed $j$, the family $(\delta_{ij} : A_i \to A_j)_{i\in I}$ corresponds to some morphism $\pi_j : A \to A_j$ such that $\pi_j\circ\kappa_i = \delta_{ij}$. Again by the isometry property, we have $\| \pi_j \|^2 = \| \sum_i \delta_{ij} \delta_{ij}^* \| = \| \id{A_j} \| \leq 1$. By Lemma~\ref{isometrylemma}, this implies $\pi_j = \kappa_j^*$. We therefore have $\kappa_j^*\kappa_i = \delta_{ij}$, as desired. This also implies that the $\kappa_j \kappa_j^*$ are mutually orthogonal projections. In order to also prove the completeness relation~\eqref{complete}, let $p : A\to A$ be any other projection that is orthogonal to each of the $\kappa_j \kappa_j^*$; we need to show $p = 0$. Again by naturality, under the isomorphism $\cat{C}(A,A) \cong \bigoplus_i \cat{C}(A_i,A)$, the new $p$ corresponds to the family $(p\kappa_i) = (0)$, where $p\kappa_i = 0$ follows from the assumed orthogonality. This implies indeed $p=0$ by the uniqueness part of the universal property.

We finally show that \ref{biprodsum} implies \ref{universalprop2}. So let~$\left(f_j: A_j \to B\right)_{j\in I}$ be any family of morphisms
        with~$\sum_{j \in I}f_jf_j^* < \infty$.
The homset~$C(B,A)$
    is a self-dual Hilbert~$\cat{C}(B,B)$-module
    with~$\cat{C}(B,B)$-valued inner
    product~$\langle g, f\rangle \equiv g^*f$~\cite[Proposition~2.15]{wstarcat}.
The self duality is equivalent to the fact
    that~$\cat{C}(B,A)$ is complete in the \emph{ultranorm topology};
    this is the topology
    generated by the seminorms~$\| f \|_\omega \equiv \omega(\langle f,f\rangle)^\frac{1}{2}$
    indexed by normal states~$\omega$ on~$\cat{C}(B,B)$~\cite[\S149\textsubscript{V}]{bas}, also known as the \emph{s-topology}~\cite[p.~88]{wstarcat}. We now turn to the construction of $f : A \to B$.
    Pick any normal state~$\omega$ on~$\cat{C}(B,B)$
        and finite subset~$S \subseteq I$.  We have
\begin{alignat*}{2}
    \bigl\| \sum_{j \in S} \kappa_j f_j^* \bigr\|^2_\omega
    &\ \equiv \ 
    \omega\Bigl(\bigl(\sum_{i \in S} \kappa_i f_i^*\bigr)^*
            \bigl(\sum_{j \in S} \kappa_j f_j^*\bigr)\Bigr) \\
    &\ =\ 
    \omega\Bigl( \sum_{j \in S} f_j f_j^*
    \Bigr) &\qquad&\text{as~$\kappa_j^*\kappa_i = \delta_{ij}$}.
\end{alignat*}
From this and~$\sum_{j \in I} f_jf_j^* < \infty$,
    it follows that the net~$(\sum_{j \in S} \kappa_j f_j^*)_S$
    is ultranorm Cauchy as $S$ ranges over all finite subsets of $I$.
    Define~$f := ( \sum_{j \in I} \kappa_jf_j^* )^*$,
        where the sum is understood to converge in the ultranorm topology.
The assignment~$g \mapsto \kappa_j^*g$
    yields a bounded~$\cat{C}(B,B)$-linear map between
    the Hilbert~$\cat{C}(B,B)$-modules~$\cat{C}(B,A)$ and~$\cat{C}(B,A_j)$
    and is therefore ultranorm continuous~\cite[\S148]{bas}.
    Hence
        $\kappa_j^* f^*
             =  \kappa_j^* \sum_{i} \kappa_i f_i^*
             =  \sum_i \kappa_{j }^*\kappa_i f_i^*
            =  f_j^*$
            and so~$f \kappa_j = f_j$ as desired.
Next we will show that~$\|f \|^2 = \|\sum_j f_j f_j^*\|$.
    By definition of~$f^*$ and \cite[\S148\textsubscript{V}]{bas}
         we have~
 \begin{equation*}
         \langle f^*, f^*\rangle
         \ =\  \uwlim_{S\subseteq I\text{ finite}} \, \Bigl\langle \sum_{j \in S} \kappa_j f^*_j,
                        \sum_{j \in S}\kappa_j f^*_j \Bigr\rangle
         \ =\  \uwlim_{S\subseteq I\text{ finite}} \, \sum_{j \in S} f_j f_j^*
        \ = \ \sum_j f_j f_j^*
 \end{equation*}
and so~$\| f \|^2 = \|f^*\|^2 = \|\langle f,f \rangle \|
                = \|\sum_j f_j f_j^* \|$ as desired.
Finally, to show uniqueness,
    assume~$f'\colon A \to B$ is any (other)
    morphism with~$f' \kappa_i = f_i$.
    Note that~$\sum_j \kappa_j \kappa_j^*$
       converges ultrastrongly
       and so in the ultranorm topology as well.
With similar reasoning as before,
        the composition map~$g \mapsto f' g$ is ultranorm continuous,
        hence
\begin{equation*}
    f'
    \ =\ f' \sum_{j \in I} \kappa_j \kappa_j^*
    \ =\ \sum_{j \in I} f' \kappa_j \kappa_j^*
    \ =\ \sum_{j \in I} f_j \kappa_j^*
    \ =\ \sum_{j \in I} f \kappa_j \kappa_j^*
    \ =\ f,
\end{equation*}
    where the sums converge in the ultranorm topology.
\end{proof}

It is now straightforward to note:

\begin{cor}
Every normal $*$-functor between W$^*$-categories preserves direct sums.
\end{cor}

Rieffel has observed this for the W$^*$-categories $\Rep{A}$ from Example~\ref{rep_ex} as~\cite[Proposition~4.9]{Rieffel1974}. The following was undoubtedly already known to the authors of~\cite{wstarcat}, but we have not found it stated anywhere explicitly. We are now in a position to give a very simple proof.

\begin{prop}
Let $N$ be a W$^*$-algebra. In the W$^*$-category $\NRep{N}$, any $\ell^2$-direct sum
\beq
\label{concrete_directsum}
	\bigoplus_{i\in I} A_i\  :=\  \Bigl\{ \textstyle (\xi_i)_{i\in I} \Bigm \vert \xi_i\in A_i,\: \sum_i \langle\xi_i,\xi_i\rangle < \infty \Bigr\},
\eeq
equipped with the componentwise representation of $N$, is an \emph{$I$-indexed direct sum} in the sense of Definition~\ref{directsumdef}. Conversely, every $I$-indexed direct sum in $\NRep{N}$ is of this form.
\end{prop}

\begin{proof}
	Given a concrete direct sum as in~\eqref{concrete_directsum}, it is straightforward to see that the inclusion maps $\kappa_i : A_i \to A$ satisy the conditions of Theorem~\ref{directsum_equiv}\ref{biprodsum}. This shows that all direct sums in $\NRep{N}$ exist; the uniqueness up to unique unitaries therefore implies that every direct sum is of this form.
\end{proof}

\bibliographystyle{plain}
\bibliography{direct_integrals}

\end{document}